\documentclass[peerreview,onecolumn]{IEEEtran}
\usepackage{amsfonts,amssymb,amsmath,amsthm}
\usepackage{xspace}
\usepackage[ruled, onelanguage, linesnumbered]{algorithm2e}
\usepackage{xcolor}
\usepackage{comment}  

\ifCLASSINFOpdf
\usepackage[pdftex]{graphicx}
\else
\fi

\newcounter{example}[section]

\newtheorem{Theorem}{Theorem }
\newtheorem{Featured Theorem}{Featured Theorem }
\newtheorem{Lemma}{Lemma }
\newtheorem*{Lemma*}{Lemma }
\newtheorem{Corollary}{Corollary }
 
\newtheorem{remark}{Remark}
\newcommand{\convexe}{\mathcal{C}\xspace} 
\newcommand{\Boule}{\mathcal{B}\xspace} 
\newcommand{\ExactSparse}{m-\textsc{Exact-Sparse}\xspace}

\newcommand{\signal}{y\xspace} 
\newcommand{\argmax}{\operatorname{arg\,max}}
\newcommand{\argmin}{\operatorname{arg\,min}}
\newcommand{\sign}{\operatorname{sign}}
\newcommand{\Span}{\operatorname{span}}
\newcommand{\conv}{\operatorname{conv}}

\newcommand{\matricesupport}{\Phi_{opt}}
\newcommand{\norme}[1]{\| {#1} \|}
\newcommand{\va}[1]{| {#1} |}
\newcommand{\ps}[1]{\langle  {#1} \rangle}

\begin{document}
%
\title{Recovery and convergence rate of the Frank-Wolfe Algorithm for the \ExactSparse Problem}
%
%
%

\author{Farah~Cherfaoui,
        Valentin~Emiya,
        Liva~Ralaivola
        and~Sandrine~Anthoine
\thanks{F. Cherfaoui, V. Emiya are with Aix Marseille Univ, Universit\'e de Toulon, CNRS, LIS, Marseille, France.}
\thanks{L. Ralaivola is with Aix Marseille Univ, Universit\'e de Toulon, CNRS, IUF and Criteo.}
\thanks{S. Anthoine is with Aix Marseille Univ, CNRS, Centrale Marseille, I2M, Marseille, France.}
}

%
%

\markboth{Journal of \LaTeX\ Class Files,~Vol.~14, No.~8, August~2015}%
{Shell \MakeLowercase{\textit{et al.}}: Bare Demo of IEEEtran.cls for IEEE Journals}
%



\maketitle

\begin{abstract}
We study the properties of the Frank-Wolfe algorithm to solve the \ExactSparse reconstruction problem, where a signal $\signal$ must be expressed as a sparse linear combination of a predefined set of atoms, called {\em dictionary}. We prove that when the signal is sparse enough with respect to the coherence of the dictionary, then the iterative process implemented by the Frank-Wolfe algorithm only recruits atoms from the support of the signal, that is the smallest set of atoms from the dictionary that allows for a perfect reconstruction of $\signal$. We also prove that under this same condition, there exists an iteration beyond which the algorithm converges exponentially.
\end{abstract}


%
\IEEEpeerreviewmaketitle


\section{Introduction}
%
%
%
%
 
\subsection{The \ExactSparse problem}


Given a signal $y$ in $\mathbb{R}^d$ that is $m$-sparse with respect to some dictionary $\Phi$ in $\mathbb{R}^{d\times n}$, i.e. $y$ is a linear combination of at most $m$ atoms/columns of $\Phi$, the \ExactSparse problem is to find a linear combination of (at most) $m$ atoms that is equal to $y$.
%
The Matching Pursuit (MP)~\cite{MZ92} and Orthogonal Matching Pursuit (OMP)~\cite{PRK93} algorithms feature nice properties with respect to the \ExactSparse problem, namely {\em recovery} and {\em convergence} properties.
Here, we study these properties for the Frank-Wolfe optimization procedure~\cite{FW56} when used to tackle the \ExactSparse problem. 

In the sequel, the dictionary $\Phi = [\varphi_1 \cdots \varphi_n] \in\mathbb{R}^{d\times n}$ is the matrix made  of the $n$ atoms $\varphi_1,\ldots,\varphi_n\in\mathbb{R}^d$, assumed to be so that $\|\varphi_i\|_2=1,\forall i$. Here, we consider the case where the \textit{support} $\Lambda \subseteq \{1,\ldots,n\}$ of the $m$-sparse signal $\signal$ is unique, the support being the smallest subset of $\{1,\ldots,n\}$ such that $\signal$ is in the span of the atoms indexed by $\Lambda$ --~therefore $ |\Lambda| \leq m$.
(Section \ref{sec:theo} gives the conditions under which the support is unique.)
The sparsity of a linear combination $\Phi x$ ($x\in\mathbb{R}^n$) such that $y=\Phi x$ is measured by the number of nonzero entries of $x$, sometimes referred to as the (quasi-)norm $\norme{x}_0$ of $x$.

Formally the  \ExactSparse problem is the following. Given a dictionary $\Phi$ and an $m$-sparse signal $\signal$:
\begin{equation}
 \text{find } x \text{ s.t. } \signal = \Phi x \text{ and } \norme{ x }_0 \leq m. \label{eq:original pb exactSparse}
 \end{equation}
Since $\signal$ is a linear combination of at most $m$ atoms of $\Phi$, a solution of Problem \eqref{eq:original pb exactSparse} can be obtained by solving:
\begin{equation}
  \label{pb exactSparse} 
   \argmin_{x \in \mathbb{R}^n } \tfrac{1}{2}\norme{\signal - \Phi x}^2_2 \ \ \text{s.t. } \norme{ x }_0 \leq m.
\end{equation}

\subsection{Related work}
The \ExactSparse problem is NP-hard~\cite{DMA97}, and the ability of a few algorithms that explicitly enforce the $\ell_0$-constraint to approximate its solution
have been studied, among which brute force methods~\cite{M02}, nonlinear programming approaches~\cite{RK97}, and greedy pursuits~\cite{MZ92, PRK93, T03, T11, BD08}.
Favorable conditions that simultaneously apply to both the sparsity level $m$ and the dictionary $\Phi$ have been exhibited that provide OMP and MP with guarantees on
their effectiveness to find exact solutions to \ExactSparse.

Another way to tackle this problem is to recourse to a relaxation strategy 
by, for instance, replacing the $\ell_0$ norm by an $\ell_1$ norm. 
Doing so gives rise to a convex optimization problem, e.g. LASSO~\cite{T96} and Basis Pursuit~\cite{CDS01}, that can be handled with a number of provably efficient methods \cite{B15}.
 In addition to the computational benefit of relaxing the problem, Tropp~\cite{T06} proved 
 that under proper conditions, the supports $\Lambda^{\text{LASSO}}$ and $\Lambda^{\text{BP}}$
 of the solutions of the LASSO and Basis Pursuit $\ell_1$-relaxations are such that $\Lambda^{\text{LASSO}}\subseteq\Lambda$ and $\Lambda^{\text{BP}}\subseteq\Lambda.$


Here we study the properties of the Frank-Wolfe algorithm \cite{FW56} to solve the following $\ell_1$-relaxation problem:
\begin{equation}\label{equation: problem ExactSparse relaxation}
\begin{split} 
  \argmin_{x \in \mathbb{R}^n } \tfrac{1}{2}\norme{\signal - \Phi x}^2_2  \ \ \text{s.t. } \norme{ x }_1 \leq \beta 
 \end{split} 
\end{equation}
where $\norme{.}_1$ is the $\ell_1$ norm and $\beta > 0$ is a hyperparameter. More precisely, we study how solving~\eqref{equation: problem ExactSparse relaxation}
may provide an exact solution to~\eqref{pb exactSparse} and, therefore, to \eqref{eq:original pb exactSparse}.

Going back to the original unrelaxed \ExactSparse problem, we may refine the evoked results
regarding the blessed  Matching Pursuit (MP)~\cite{MZ92} and Orthogonal Matching Pursuit (OMP)~\cite{PRK93} greedy algorithms: 
Tropp~\cite{T04} and Gribonval and Vandergheynst~\cite{GV06} proved that, if the size $m$ of the support is small enough, then at each iteration, both MP and OMP pick up an atom indexed by the support, thus ensuring recovery properties. Furthermore, as far as convergence is concerned,
that is as far as the ability of the procedures to find an exact linear expansion of $\signal$ 
is concerned,
it was proved that MP shows an exponential rate of convergence, and that OMP reaches convergence after exactly $m$ iterations.
\begin{remark}
Tropp~\cite{T04} and Gribonval and Vandergheynst~\cite{GV06} also generalized their results to 
the case of compressible signals, which are signals that are close to be $m$-sparse. It is beyond
the scope of the present paper to consider such signals.
\end{remark}

On the other hand, the Frank-Wolfe algorithm~\cite{FW56} is a general-purpose algorithm designed for constrained convex optimization. It has been proven to converge exponentially if the objective function under consideration is strongly convex~\cite{GM86} and linearly in the other cases~\cite{FW56}.
As we will see in Section \ref{section: FW for ExactSparse}, when solving~\eqref{equation: problem ExactSparse relaxation}, the atom selection steps in Matching Pursuit and Frank-Wolfe are very similar. This similarity inspired Jaggi and al.~\cite{LKTJ17} to exploit tools  used to analyze the Frank-Wolfe algorithm and prove the convergence of the MP algorithm when no assumptions are made on the dictionary.

Here, we go the other way around: we will exploit tools used to analyze MP and OMP to establish  properties of the Frank-Wolfe algorithm when seeking a solution of Problem \eqref{equation: problem ExactSparse relaxation}. 

\subsection{Main Results}
We show that the Frank-Wolfe algorithm, when 
used to solve~\eqref{equation: problem ExactSparse relaxation}, enjoys recovery and convergence properties regarding \ExactSparse that are
similar to those established in~\cite{GV06,T04} for MP and OMP, under the very same assumptions.

Our results rely on a fundamental quantity associated to a dictionary $\Phi = [\varphi_1 \cdots \varphi_n ]$: its Babel function, defined as
$$ \mu_1(m) = \max\limits_{\va{ \Lambda } = m}\max\limits_{i \notin \Lambda} \sum\limits_{j \in \Lambda} \va{ \ps{  \varphi_i, \varphi_{j} } },$$
where, by convention:  $\mu_1(0) = 0$. 
For given $m$, $\mu_1(m)$ is roughly a measure on how well any atom from $\Phi$ can be expressed as a linear combination of a set of $m$ other atoms.
When $m = 1$, the Babel function boils down to
$$ \mu = \mu_1(1) = \max\limits_{j \neq k} \va{ \ps{  \varphi_{j}, \varphi_{k} } },$$
which is known as the {\em coherence} of $\Phi$.
In the sequel, we will consider only $m$-sparse signal such that: $ m < \tfrac{1}{2}(\mu^{-1} + 1)$.
In Featured Theorem~\ref{Featured Theorem: recovery condition FW}, we state that when $ m < \tfrac{1}{2}(\mu^{-1} + 1)$, then at each iteration the Frank-Wolfe algorithm picks up an atom indexed by the support of $\signal$.
\begin{Featured Theorem}\label{Featured Theorem: recovery condition FW}
Let $\Phi$ be a dictionary of coherence $\mu$, and $\signal$ an $m$-sparse signal such that $ m < \tfrac{1}{2}(\mu^{-1} + 1)$. 
Then at each iteration, the Frank-Wolfe algorithm picks up an atom of the support of the signal.
\end{Featured Theorem}
Under the same condition, we also prove that the rate of convergence of the Frank-Wolfe algorithm is exponential beyond a certain iteration even though the function we consider is not strongly convex. This is given by Featured Theorem~\ref{Featured Theorem: exponential convergence FW}.
\begin{Featured Theorem}\label{Featured Theorem: exponential convergence FW} 
Let $\Phi$ be a dictionary of coherence $\mu$, and $\signal$ an $m$-sparse signal such that $ m < \tfrac{1}{2}(\mu^{-1} + 1)$. Under some conditions on $\signal$ and $\beta$, there exists an iteration $K$ of the Frank-Wolfe algorithm and $0 < \theta \leq 1$ such that:
$$\norme{\signal - \Phi x_{k+1}}^2_2 \leq \norme{\signal - \Phi x_k}^2_2 (1 - \theta) \qquad  \forall \ k \geq K$$
where $\theta$ depends on $\mu_1(m-1), \ \beta$, and $\signal$ and $0 < \theta \leq 1$
(which implies the exponential convergence).
\end{Featured Theorem}
\subsection{Organization of the Paper}
In Section~\ref{sec:algo}, we instantiate the Frank-Wolfe algorithm for Problem \eqref{equation: problem ExactSparse relaxation} and we relate it to MP and OMP. Section~\ref{sec:theo} is devoted
to the statement of our main results.  We probe their optimality with numerical experiments in Section~\ref{sec:expe}.


\section{MP, OMP and FW Algorithms}
\label{sec:algo}

This section recalls Matching Pursuit and Orthogonal Matching Pursuit, the classical greedy algorithms used to tackle \ExactSparse.
 We then present the Frank-Wolfe algorithm and derive it for Problem \eqref{equation: problem ExactSparse relaxation}, showing its similarities with MP and OMP.

\subsection{Matching Pursuit and Orthogonal Matching Pursuit}
Let $\Phi$ be an orthonormal basis and $\signal$ an exactly $m$-sparse signal (i.e. $\signal = \Phi x$ with $x$ having  exactly $m$ nonzero entries). In this case, the signal can be expressed as $$\signal = \sum\limits_{\stackrel{\lambda \in \Lambda}{|\Lambda|=m}} \ps{\signal, \varphi_{\lambda}} \varphi_{\lambda},$$ where $\Lambda$ is the index set of the $m$ atoms $\varphi$ that satisfy: $\ps{\signal, \varphi} \neq 0$.
The \ExactSparse problem has then an easy solution: one chooses the $m$ atoms having the nonzero inner products with the signal, and 
the linear expansion of $\signal$ with respect to these atoms can be obtained readily. 

Algorithmically, this can be achieved by building $\signal_k$, the approximation of $\signal$, one term at a time. Noting $y_k$ the current
 approximation and $r_k = \signal - \signal_k$ the so-called {\em residual}, we select at each time step the atom which has the largest inner product (this is a greedy selection) with $r_k$, and update the approximation. 

MP~\cite{MZ92} and OMP~\cite{PRK93} are two greedy algorithms used for approximating signals in the general case where the dictionary is not an orthonormal basis. They build upon this idea of greedy selection and iterative updates. MP and OMP initialize the first approximation $\signal_0 = 0$ and residual $r_0 = \signal$ and then repeat the following steps:
\begin{enumerate}
  \item Atom selection: $\quad\lambda_{k} = \argmax_{i} \va{ \ps{\varphi_i, r_k} }.$
  \item Approximation update:
  \begin{enumerate}
    \item MP: $\quad\signal_{k+1} = \signal_{k} + \ps{\varphi_{\lambda_{k}}, r_k}\varphi_{\lambda_{k}}$
    \item OMP: $\quad \signal_{k+1} = \argmin_{a \in \Span(\{\varphi_{\lambda_{0}}, \dots, \varphi_{\lambda_{k}} \})} \| \signal - a \|_2 $
  \end{enumerate}
  \item Residual update: $\quad r_{k+1} = \signal - \signal_{k+1}.$
\end{enumerate}

\subsection{Frank-Wolfe}
The Frank-Wolfe algorithm~\cite{FW56} is an iterative algorithm developed to solve the optimization problem:
\begin{equation}
\label{eq:FW}
\begin{split} 
  \min_{x \in \convexe} f(x)  \ \ \text{s.t. } x \in \convexe 
 \end{split} 
\end{equation}
where $f$ is a convex and continuously differentiable function and $\convexe$ is a compact and convex set. 
The Frank-Wolfe algorithm is initialized with an element of $\convexe$. Then, at iteration $k$, 
the algorithm applies the three following steps:
\begin{enumerate}
  \item Descent direction selection: $\quad s_k = \argmin_{s \in \convexe} \ps{s, \nabla f(x_k)}.$
  \item Step size optimization:  $\quad \gamma_k = \argmin_{\gamma \in [0, 1]} f((1-\gamma)x_k + \gamma s_k).$
  \item Update: $\quad x_{k+1} = (1-\gamma_k)x_k + \gamma_k s_k.$
\end{enumerate}
Note that the step-size $\gamma_k$ can be chosen by other methods~\cite{LKTJ17}, without affecting the convergence properties of the algorithm.

\subsection{Frank-Wolfe for the  \ExactSparse problem}
\label{section: FW for ExactSparse}
We are interested in solving  \ExactSparse  by finding the solution of the following problem:
\begin{equation}\tag{\ref{equation: problem ExactSparse relaxation}}
\begin{split} 
  \argmin_{x \in \mathbb{R}^n } \tfrac{1}{2}\norme{\signal - \Phi x}^2_2  \ \ \text{s.t. } \norme{ x }_1 \leq \beta 
 \end{split} 
\end{equation}
using the Frank-Wolfe algorithm. To this end, we instantiate~\eqref{eq:FW} for Problem~\eqref{equation: problem ExactSparse relaxation}:
$$\min_{x \in \convexe} f(x) = \tfrac{1}{2}\norme{\signal - \Phi x}^2_2,$$
%
$$\text{s.t. } x \in \Boule_1(\beta) = \{ x: \ \norme{x}_1 \leq \beta \}.$$ 
$\Boule_1(\beta)$ is the $\ell_1$ ball of radius $\beta$; it can be written as the convex hull: $\Boule_1(\beta) = \conv\{ \pm \beta e_i | i \in \{1,\dots,n \} \}$, with $e_i$ the canonical basis vectors of $\mathbb{R}^n$. Moreover, $\nabla f(x) = \Phi^{t}(\Phi x - \signal)$. The 
selection step of the Frank-Wolfe algorithm thus becomes:
\begin{align*}
s_k &= \argmin_{s \in \conv\{ \pm \beta e_i | i \in \{1,\dots,n \} \}} \ps{s, \Phi^{t}(\Phi x_k - \signal)}.
\end{align*}
Since this optimization problem is linear and $\Boule_1(\beta)$ is closed and bounded, there is always an extreme point of $\Boule_1(\beta)$ in the solution set (see~\cite{BV04} for more details), thus:
\begin{align*}
 s_k &= \argmin_{s \in \{ \pm \beta e_i | i \in \{1,\dots,n \} \}} \ps{s, \Phi^{t}(\Phi x_k - \signal)}
\end{align*}
or equivalently
\begin{align*}
s_k &= \argmax_{s \in \{ \pm \beta e_i | i \in \{1,\dots,n \} \}} \ps{\Phi s, \signal-\Phi x_k}.
\end{align*}
Noticing that  $s = \pm \beta e_i$ implies $\Phi s = \pm \beta \varphi_{i}$, we conclude that the direction selection step of the Frank-Wolfe algorithm for Problem \eqref{equation: problem ExactSparse relaxation} can be rewritten as:
\begin{align*} 
\begin{cases}
i_k &= \argmax_{i \in \{1,\dots,n \} } | \ps{  \varphi_i, \signal - \Phi x_k } | \\
s_k &= \sign(\ps{ \varphi_{i_k},\signal - \Phi x_k  })\beta e_{i_{k}}.
\end{cases}
\end{align*}
Recalling that the residual $r_k$ is: 
$$r_k = \signal - \Phi x_k,$$ 
we notice that we have the same atom selection as in MP and OMP: 
$$i_k=\lambda_k=\argmax_{i} | \ps{  \varphi_i, r_k } |. $$
Finally, we specify the initialization $x_0 = 0$ which is in $\Boule_1(\beta)$. This completes the description of the Frank-Wolfe algorithm for Problem \eqref{equation: problem ExactSparse relaxation} which is summarized in Algorithm~\ref{algorithm: Frank-Wolfe for ExactSparse}.
\begin{algorithm}[t] 
 \KwData{signal $\signal$, dictionary $\Phi = [\varphi_1, \dots, \varphi_n], \beta>0$.}
 $x_0 = 0 $\\
 $k = 0$\\
 \While{stopping criterion not verified}{
 $i_{k} = \argmax_{i \in \{1,\dots,n \} } | \ps{  \varphi_i,\signal - \Phi x_k  } | $\\ 
 $s_k = \sign(\ps{ \varphi_{i_k},\signal - \Phi x_k  })\beta e_{i_{k}}$\\
 $\gamma_k = \argmin_{\gamma \in [0,1]} \norme{ \signal - \Phi (x_k + \gamma(s_k - x_k)) }_2^2$ \\
 $x_{k+1} = x_k + \gamma_k(s_k - x_k)$\\
 $k = k + 1$\\
 }
 \caption{The Frank-Wolfe algorithm for Problem \eqref{equation: problem ExactSparse relaxation}}
 \label{algorithm: Frank-Wolfe for ExactSparse}
\end{algorithm}

In the sequel, we will be interested in the recovery and convergence properties of this algorithm when $ m < \tfrac{1}{2}(\mu^{-1} + 1)$. 
This hypothesis implies that the atoms  of any subset of at most $m$ atoms ($\{\varphi_i | i \in \Lambda\}$ such that $|\Lambda| \leq m$) are necessarily linearly independent and also that for any $m$-sparse signal $\signal$, the expansion coefficients $x^*$ such that $\signal = \Phi x^*$ and $\norme{x^*}_0\leq m$ and the corresponding support are unique \cite{T04}.

In that case, $x^*$ is the unique solution of the \ExactSparse problem and also of Problem~\eqref{pb exactSparse} but not always a solution of  Problem~\eqref{equation: problem ExactSparse relaxation}. Indeed, we always have  $f(x^*)=0$, so if $\norme{x^*}_1 \leq \beta$ then  $x^*$ is a solution of  Problem~\eqref{equation: problem ExactSparse relaxation}, but if $\norme{x^*}_1 > \beta$ then  $x^*$ is not feasible for Problem~\eqref{equation: problem ExactSparse relaxation} hence it is not a solution.

 
Now, let us clarify some notations. For an $m$-sparse signal $\signal = \Phi x^*$, we denote by $\Lambda_{opt}$ its support i.e. $\signal = \sum_{i \in \Lambda_{opt}} x^*[i] \varphi_i$ such that $|\Lambda_{opt}| = m$.
For $\Lambda$ a subset of $\{1, \dots,n \}$, we denote by $\Phi_{\Lambda}$ the matrix whose columns are the atoms indexed by $\Lambda$. When $\Lambda$ is the support $\Lambda_{opt}$, we note $\lambda_{min}^{*}$ (resp. $\lambda_{max}^{*}$) its lowest (resp. largest) singular value. 
For a matrix $\Phi$ we denote by $\Span(\Phi)$ the vector space spanned by its columns. Finally, when we study the convergence of Algorithm~\ref{algorithm: Frank-Wolfe for ExactSparse}, we consider the residual squared norm $\norme{r_k}^2_2$, tied to the objective function as follows: 
$$ f(x_k) = \tfrac{1}{2} \norme{\signal - \Phi x_k}^2_2 =  \tfrac{1}{2} \norme{r_k}^2_2. $$


\section{Results: Exact recovery and exponential convergence}
\label{sec:theo}

In this section, we state our main results on the recovery property and the convergence rate of Algorithm~\ref{algorithm: Frank-Wolfe for ExactSparse}. We state in Theorem~\ref{theorem: recovery for FW} the recovery guarantees of this algorithm, and we present its convergence rate in Theorem~\ref{theorem: exponential convergence FW}.

\subsection{Recovery condition}
Tropp~\cite{T04} proved that when $ m < \tfrac{1}{2}(\mu^{-1} + 1)$, then OMP exactly recovers the $m$-expansion of any $m$-sparse signal. 
Gribonval and Vandergheynst~\cite{GV06} proved that the approximated signal constructed by MP algorithm converges to the initial signal. 
To do so, they prove that at each step, MP and OMP select an atom of the support. Theorem~\ref{theorem: recovery for FW} extends this result to the Frank-Wolfe algorithm.

\begin{Theorem}
\label{theorem: recovery for FW}
Let $\Phi$ be a dictionary, $\mu$ its coherence, and $\signal$ an $m$-sparse signal of support $\Lambda_{opt}$. 
If $ m < \tfrac{1}{2}(\mu^{-1} + 1) $, then at each iteration, Algorithm~\ref{algorithm: Frank-Wolfe for ExactSparse} picks up a correct atom, i.e. $\forall \ k$, $ i_k \in \Lambda_{opt}$.
\end{Theorem}

\begin{remark}[ERC]
As in~\cite{T04,GV06}, the condition $ m < \tfrac{1}{2}(\mu^{-1} + 1) $ can be replaced by a 
support-specific condition called the exact recovery condition (ERC): $ \max\limits_{i \notin \Lambda_{opt}} \| \Phi_{\Lambda_{opt}}^{+} \varphi_i \|_1 < 1 $, where $\Phi_{\Lambda_{opt}}^{+}$ is the pseudoinverse of the matrix $\Phi_{\Lambda_{opt}}$.
ERC is not easy to check in practice because it depends on the unknown support $\Lambda_{opt}$, while the condition
 $m < \tfrac{1}{2}(\mu^{-1} + 1) $, a sufficient condition for ERC~\cite{T04} to hold, is easy to check.
\end{remark}

\begin{proof}[Proof of Theorem~\ref{theorem: recovery for FW}]
The proof of this theorem is very similar to the proof of Theorem 3.1 in~\cite{T04}.
One shows by induction that at each step the residual $r_k = \signal - \Phi x_k$ remains in $\Span(\Phi_{\lambda_{opt}})$ and in the process that the selected atom is in $\Lambda_{opt}$. 
\begin{itemize}
  \item $k = 0$: by definition $r_0 = \signal$ is in $\Span(\Phi_{\lambda_{opt}})$.
  \item If $k \geq 0$: we assume that $r_k \in \Span(\Phi_{\Lambda_{opt}})$. Let $\lambda = \{1, \dots, n\} \setminus \lambda_{opt}$ be the set of atoms which are not in the support of the signal.
  The atom $\varphi_{i_k}$ is a ``good atom'' (i.e. $i_{k} \in \Lambda_{opt}$), if and only if:
$$ \rho(r_k) =  \frac{\norme{ \Phi_{\lambda}^t r_k }_{\infty}}{\norme{\matricesupport^t r_k }_{\infty}} < 1.$$
Tropp~\cite{T04} proved that if $r_k \in \Span(\Phi_{\Lambda_{opt}})$ then $ \rho(r_k) \leq \max\limits_{i \notin \Lambda_{opt}} \| \Phi_{\Lambda_{opt}}^{+} \varphi_i \|_1 \leq \frac{\mu_1(m)}{1 - \mu_1(m-1)}$, where $ \Phi_{\Lambda_{opt}}^{+}$ is the pseudoinverse of the matrix $ \Phi_{\Lambda_{opt}}$. He also proved that $ m < \tfrac{1}{2}(\mu^{-1} + 1) $ implies $ \frac{\mu_1(m)}{1 - \mu_1(m-1)}<1$.

Hence $i_k$ is in $\Lambda_{opt}$ and $s_k=\pm\beta e_{i_k}$ is thus in $\Span(\Phi_{\Lambda_{opt}})$. Since $r_{k+1} = r_k - \gamma_k \Phi (s_k - x_k)$, and since by assumption $r_k$ is also in $\Span(\Phi_{\Lambda_{opt}})$, we deduce that $r_{k+1}$ is in $\Span(\Phi_{\Lambda_{opt}})$.
\end{itemize}

\end{proof}

Theorem~\ref{theorem: recovery for FW} specifies that if the signal has a sparsity $m$ smaller than $\tfrac{1}{2}(\mu^{-1} + 1) $, Algorithm~\ref{algorithm: Frank-Wolfe for ExactSparse} only recruits atoms of the support. As noted is Section \ref{section: FW for ExactSparse}, the expansion $x^*$ can not always be reached (because it might be the case that $\norme{x^*}_1 > \beta$). In the case when it can be reached (i.e. when $\norme{x^*}_1 \leq \beta$) one can furthermore prove that the expansion $x^*$ itself is recovered:

\begin{Corollary}
\label{cor:convergence}
Let $\Phi$ be a dictionary, $\mu$ its coherence, and $\signal$ an $m$-sparse signal of support $\Lambda_{opt}$. 
If $ m < \tfrac{1}{2}(\mu^{-1} + 1) $ and $\norme{x^*}_1 \leq \beta$ then the sequence $x_k$ converges to $x^*$ (i.e. Algorithm~\ref{algorithm: Frank-Wolfe for ExactSparse} exactly recovers the $m$-expansion of $\signal$).
\end{Corollary}

\begin{proof}[Proof of Corollary~\ref{cor:convergence}]
$\norme{x^*}_1 \leq \beta$ so that $x^*$ is a solution of Problem~\eqref{equation: problem ExactSparse relaxation}. The Frank-Wolfe algorithm is known to converge in terms of objective values ($f(x_k)$), we deduce that $f(x_k)$ converges to  $f(x^*)=0$. Since Theorem~\ref{theorem: recovery for FW} ensures that the iterates $x_k$ are in $\Span(\Phi_{\Lambda_{opt}})$, we also have convergence of the iterates ($x_k$ converge to $x^*$) since
\begin{align*} 
     |f(x_k)-f(x^*)| & =  |f(x_k) - 0| = \tfrac{1}{2} \norme{\signal - \Phi x_k}_2^2 \\ 
     & = \tfrac{1}{2} \norme{\Phi x^* - \Phi x_k}_2^2 \\
     & \geq \tfrac{1}{2} \left(\lambda_{min}^{*}\right)^2 \|x^*- x_k\|_2^2 
\end{align*}
where the last line holds because $x_k-x^*$ is in $\Span(\Phi_{\Lambda_{opt}})$
and $\lambda_{min}^{*}>0$ since  $ m < \tfrac{1}{2}(\mu^{-1} + 1) $~\cite{T04}. 
\end{proof}
In this section, we presented the recovery guarantees for the Frank-Wolfe algorithm. In the next section, we will show that the convergence rate of the Frank-Wolfe algorithm is exponential when $ m < \tfrac{1}{2}(\mu^{-1} + 1) $ and $\beta$ is large enough so that the expansion $x^*$ is recovered.

\subsection{Rate of convergence}
As mentioned in the introduction, in the generic case of Problem~\eqref{eq:FW}, the Frank-Wolfe algorithm converges exponentially beyond a certain iteration when the objective function is strongly convex~\cite{GM86} and linearly in the other cases~\cite{FW56}. We prove in Theorem~\ref{theorem: exponential convergence FW}, that when $ m < \tfrac{1}{2}(\mu^{-1} + 1) $, the Frank-Wolfe algorithm converges exponentially beyond a certain iteration even though the function we consider is {\em not} strongly convex.

\begin{Theorem}
\label{theorem: exponential convergence FW}
Let $\Phi$ be a dictionary, $\mu$ its coherence, $\mu_1$ its Babel function, and $y = \Phi x^*$ an $m$-sparse signal. \\
If $ m < \tfrac{1}{2}(\mu^{-1} + 1) $ and $\norme{x^*}_1 < \beta$, then there exists $K$ such that for all iterations $k \geq K$ of Algorithm~\ref{algorithm: Frank-Wolfe for ExactSparse}, we have:
$$ \norme{ r_{k+1} }_2^2 \leq \norme{ r_{k} }_2^2 (1 - \theta)$$
where 
$$\theta = \tfrac{1}{16}\left(\tfrac{1 - \mu_1(m-1)}{m}\right)\left(1 - \tfrac{\norme{x^*}_1}{\beta}\right)^2 .$$
\end{Theorem}

The proof of Theorem~\ref{theorem: exponential convergence FW} is available in Appendix~\ref{sec:proofTh2}.

\begin{remark}
Note that $0 < \theta \leq 1$. Indeed, $ m < \tfrac{1}{2}(\mu^{-1} + 1) $ implies $0<\mu_1(m-1)<1$ so that $ 0 < \tfrac{1}{16m}(1 - \mu_1(m-1))\left(1 - \tfrac{\norme{x^*}_1}{\beta}\right)^2 \leq 1$ i.e. $0 < \theta \leq 1$. Thus, Theorem~\ref{theorem: exponential convergence FW} shows exponential convergence.
\end{remark}
\begin{remark}
As for Theorem~\ref{theorem: recovery for FW}, the same result holds when the Exact Recovery Condition (ERC), $\mu_1(m-1) < 1$ and $\norme{x^*}_1<\beta$ hold.
\end{remark}
\begin{remark}
As we said above, the objective function that we consider is not strongly convex, but since the constructed iterates $x_k$ remains in the $\Span (\Phi_{\Lambda_{opt}})$, the function takes the form of a $\lambda_{min}^{*}$-strongly convex function.
\end{remark}

A natural question that comes from  Theorem~\ref{theorem: exponential convergence FW} and from the convergence rate of MP and OMP is whether it is possible to guarantee the exponential convergence from the first iteration. The following theorem proves that this is possible if $\beta$ is large enough.

\begin{Theorem}
\label{theorem: minimal value for beta}
Let $\Phi$ be a dictionary of coherence $\mu$ and $y$ an $m$-sparse signal. If $m < \tfrac{1}{2}(\mu^{-1} + 1)$ and 
$$\beta > 2\norme{\signal}_2\sqrt{\tfrac{m}{1 - \mu_1(m-1)}} ,$$
then Algorithm~\ref{algorithm: Frank-Wolfe for ExactSparse} converges exponentially from the first iteration.
\end{Theorem}

The proof of Theorem~\ref{theorem: minimal value for beta} is also  available in Appendix~\ref{sec:proofTh3}.

In the proof of Theorem~\ref{theorem: exponential convergence FW} (see  Appendix~\ref{sec:proofTh2}), we show that when the iterates $x_k$ stay close enough to $x^*$ ($\norme{x_k-x^*}_1\leq \epsilon = \frac{\beta-\norme{x^*}_1}{2}$), Algorithm~\ref{algorithm: Frank-Wolfe for ExactSparse} converges exponentially. The intuition of Theorem~\ref{theorem: minimal value for beta} is to choose $\beta$ large enough so that a similar bound is guaranteed from the first iteration.

\begin{remark}
\label{rk:hyp}
Let us remark that the assumption $\beta > 2\norme{\signal}_2\sqrt{\tfrac{m}{1 - \mu_1(m-1)}} $ is stronger than that of Theorem~\ref{theorem: exponential convergence FW} ($\beta > \norme{x^*}_1$). Indeed, we have on the one hand:
$  \norme{x^*}_1 \leq\sqrt{m}\norme{x^*}_2$  because $x^*$ has non-zero coefficients only in $\Lambda_{opt}$. On the other hand:
\begin{equation*}
\norme{\signal}_2 = \norme{\Phi x^*}_2 \geq \lambda^*_{min}\norme{x^*}_2 \geq \sqrt{1 - \mu_1(m-1)}\norme{x^*}_2.
\end{equation*}
We conclude
\begin{equation}
\norme{x^*}_1 \leq\norme{\signal}_2\sqrt{\tfrac{m}{1 - \mu_1(m-1)}}.
\end{equation}
So that $\beta > 2\norme{\signal}_2\sqrt{\tfrac{m}{1 - \mu_1(m-1)}}$ implies $\beta >\norme{x^*}_1$.

Besides, while the assumption of Theorem~\ref{theorem: exponential convergence FW} ($\beta > \norme{x^*}_1$) can not be verified beforehand since it depends on the unknown $x^*$, the assumption of Theorem~\ref{theorem: minimal value for beta} can be checked since it depends on the dictionary and $\signal$.
\end{remark}


\section{Numerical Simulations}
\label{sec:expe}

Theorem~\ref{theorem: exponential convergence FW} shows that there is exponential convergence when the sparsity $m$ is small enough ($m < \tfrac{1}{2}(\mu^{-1} + 1)$) and $\beta$ is larger than $\norme{x^*}_1$. The goal of this section is to investigate whether these conditions are tight by performing three numerical experiments on synthetic data.

We simulate  in Python signals of size $d=10000$ that are sparse on a dictionary of $n=20000$ atoms. The considered dictionaries are random, with coefficients following a standard normal distribution.
The mean coherence of such dictionaries is around $\mu = 5.8 \times 10^{-2}$, and $m^* = \lceil \tfrac{1}{2}(\mu^{-1} - 1)\rceil = 8$ is the largest integer such that the condition of Theorem~\ref{theorem: exponential convergence FW} holds. The signals are also random. The supports of size $m$ are drawn uniformly at random while the corresponding coefficients are chosen using a standard normal distribution. For each experiment, Algorithm~\ref{algorithm: Frank-Wolfe for ExactSparse} is run for $2000$ simulated signals and dictionaries.

The exponential convergence in Theorem~\ref{theorem: exponential convergence FW}, is quantified by
$$ \norme{ r_{k} }_2^2 \leq \norme{ r_{k-1} }_2^2 (1 - \theta). $$
which is equivalent to
\begin{equation*}\label{eqt: log residu}
  \log \norme{ r_{k} }_2 \leq \log\norme{ \signal } + \tfrac{1}{2} k\log (1 - \theta).
\end{equation*} 
In the first two experiments, we visualize the convergence rate by displaying the quantity $\log \norme{ r_{k} }_2$, the convergence being exponential when this is upper-bounded by a line with negative slope (the steepest the slope, the fastest the convergence).

In the first experiment, the values of $m$ and $\beta$ comply with the conditions of Theorem~\ref{theorem: exponential convergence FW}. We fix $\beta = 8 \norme{x^*}_1$, and $m = m^*$. We draw in Figure~\ref{figure: exp1 convergence expo} the mean and the maximum over the $2000$ simulated signals and dictionaries of the function $\log \norme{r_{k}}_2$ for each iteration $k$, and compare it to the theoretical bound in Theorem~\ref{theorem: exponential convergence FW}. As expected, the maximum and the mean of the function $\log \norme{r_{k}}_2$ can be bounded above by a line with negative slope, and thus converge exponentially. We also notice that in practice, the maximum and the mean are much lower than the theoretical prediction. This suggests that the theoretical bound might be improved in this case.
\begin{figure}[h]
\includegraphics[scale = 0.45]{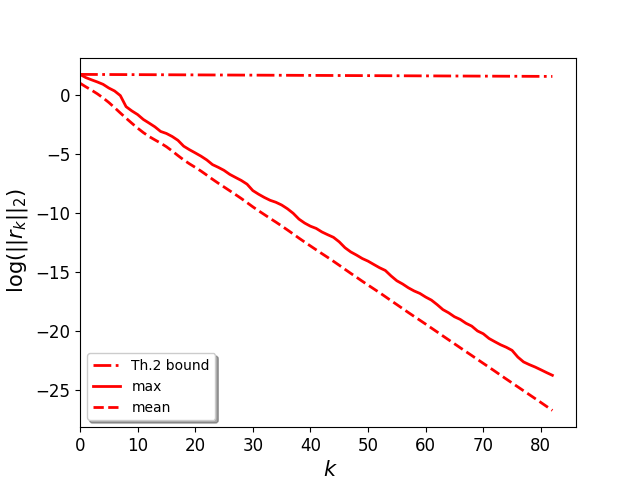}
\caption{Comparison of $\log\norme{r_k}_2$ with the theoretical bound on $2000$ simulations for $\beta = 8 \norme{x^*}_1$ and $m =m^*$.}
\label{figure: exp1 convergence expo}
\end{figure}

In the second experiment, we investigate if the exponential convergence is still possible when the sparsity is larger than $m^* =\lceil \tfrac{1}{2}(\mu^{-1} + 1)\rceil - 1$, i.e., when the condition of Theorem~\ref{theorem: exponential convergence FW} is not satisfied. 
We fix here $\beta = 8\norme{x^*}_1$ and show in Figure~\ref{figure: exp2 different value of m} the maximal value of $\log \norme{r_{k}}_2$ for $m=m^*$, $2m^*$, $5m^*$ and $20m^*$ on $2000$ signals and dictionaries. We observe that exponential convergence still arises at least up to $m=5m^*$ but probably not for $m = 20m^*$, suggesting that in practice one may reconstruct very fast a larger set of signals than only those being $m^*$-sparse, and that there might be room for a little improvement in the assumption $m \leq m^*$ in Theorem~\ref{theorem: exponential convergence FW}.
\begin{figure}[h]
\includegraphics[scale = 0.45]{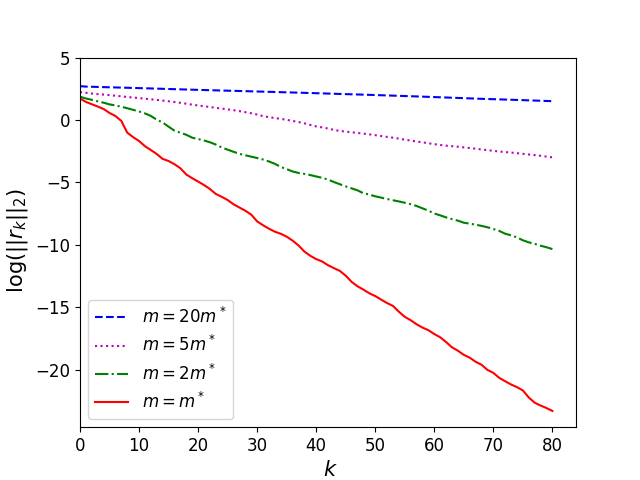}
\caption{Influence of $m$ on the maximum value of $\log\norme{r_k}_2$ on $2000$ simulations for $\beta = 8\norme{x^*}_1$.}
\label{figure: exp2 different value of m}
\end{figure}

In the last experiment, we study the influence of the distance from $x^*$ to $\Boule_1(\beta)$ on the convergence rate. Indeed Theorem~\ref{theorem: exponential convergence FW} predicts that the convergence slows down when $\norme{x^*}_1$ approaches $\beta$ and does not predict exponential convergence if $\norme{x^*}_1=\beta$.
In this experiment, the sparsity $m$ is fixed to $m =m^*$. 
We show in Figure~\ref{figure: exp3 different value of beta log} the mean and theoretical values of $\log\norme{r_k}_2$
on $2000$ signals and dictionaries
 in two cases: either $\beta = \beta_1 = 1.1\norme{x^*}_1$ or $\beta = \beta_2 = 8\norme{x^*}_1$. 
 As expected, the negative slope is steeper when $\beta$ is larger. For large $\beta$ ($\beta_2$), the slope stays well below the slope predicted by the theoretical bound. This is not the case anymore for $\beta$ close to $\norme{x^*}_1$  ($\beta_1$), where the curve becomes horizontal, suggesting that the theoretical bound may be reached and that the assumption $\beta>\norme{x^*}_1$ might be necessary.


%
\begin{figure}[h]
\includegraphics[scale = 0.45]{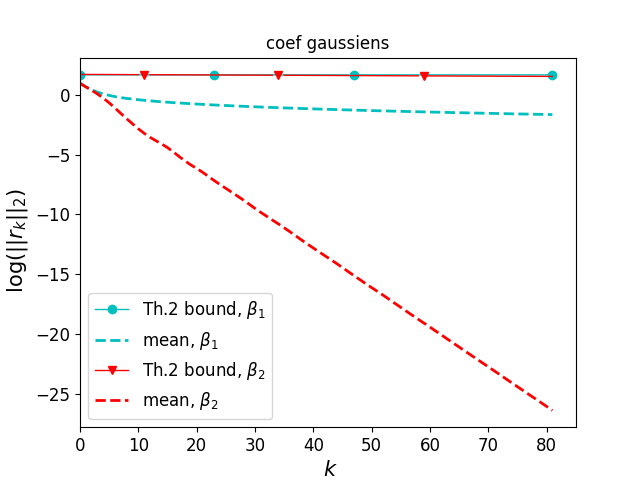}
\caption{Mean of $\log\norme{r_k}_2$
on $2000$ simulations
 for $m = m^*$  and $\beta_1 = 1.1\norme{x^*}_1$ 
 and $\beta_2 = 8\norme{x^*}_1$. 
 }
\label{figure: exp3 different value of beta log}
\end{figure}


\section{Conclusion}
We study the properties of the Frank-Wolfe algorithm  when solving the \ExactSparse problem
and we prove that, like with MP and OMP, when the signal is sparse enough with respect to the coherence of the dictionary, the Frank-Wolfe algorithm picks up only atoms of the support. We also prove that under this same condition, the Frank-Wolfe algorithm converges exponentially. 
In the experimental part, we have observed the optimality of the obtained bound in terms of the size of the $\ell_1$-ball constraining the search space, gaining some insights on the sparsity bound, that would suggest studying its tightness in future work. Extending these results to the case of non-exact-sparse but only compressible signals is also a  natural next step.

\appendices

\section{Proof of Theorem \ref{theorem: exponential convergence FW}}
\label{sec:proofTh2}
To prove Theorem \ref{theorem: exponential convergence FW}, we need to introduce the following lemma:
\begin{Lemma} 
\label{lemma: expression of gamma}
For any iteration $k$ of Algorithm~\ref{algorithm: Frank-Wolfe for ExactSparse}, if $\gamma_k \neq 0$, and if $\norme{ y }_2 < \beta$ then:
\begin{equation*}
  \gamma_k = \frac{\ps{ \Phi(s_k - x_k) , r_k }}{\norme{\Phi(s_k - x_k) }^2_2}. \nonumber 
\end{equation*}
\end{Lemma}

\begin{proof}
Recall that 
$$\gamma_k = \argmin_{\gamma \in [0, 1]} \norme{\signal - \Phi(x_k + \gamma(s_k - x_k) )}^2_2 $$
and define
$$\gamma^*_k = \argmin_{\gamma \in \mathbb{R}} \norme{ \signal - \Phi(x_k + \gamma(s_k - x_k) }^2_2.$$
Note that 
$$\gamma^*_k=\frac{\ps{ \Phi(s_k - x_k) , r_k }}{\norme{\Phi(s_k - x_k) }^2_2},$$
so we wish to prove that $\gamma_k=\gamma^*_k$.

Because $ \gamma_k$ is the solution of the same minimization problem as $\gamma^*_k$ but restricted on the interval $[0, 1]$, we have only three possibilities: (i) $\gamma^*_k\geq 1$ and $\gamma_k= 1$, (ii)
$0<\gamma_k=\gamma^*_k<1$, (iii) $\gamma^*_k\leq 0$ and $\gamma_k= 0$.
Here we assume that $\gamma_k\not= 0$ so the last possibility (iii) is ruled out.
What is left to do to finish the proof is to rule out the first possibility: (i) $\gamma^*_k\geq 1$ and $\gamma_k= 1$.

To do so, consider these two different cases:
\begin{itemize}
  \item $k = 0$: since $x_0=0$ and $r_0=\signal$,
  $$\gamma^*_0 =
\frac{\ps{  \Phi s_0, y }}{ \| \Phi s_0 \|_2^2 }  \leq  \frac{\| y \|_2}{\beta}. $$
Since $\norme{ y }_2 < \beta$, we have $\gamma_0^* < 1$. Moreover, by construction of $s_0$, $\ps{  \Phi s_0, y }>0$  so $0<\gamma_0^* < 1$. We conclude we are in case (ii) and $\gamma_0=\gamma_0^*$.
  \item $k \neq 0$: assume that $\gamma_k = 1$. We then have $x_{k+1} = s_k$.
By construction of the Frank-Wolfe algorithm we have:
$$ f(x_{k+1}) = f(s_k) \leq f(x_k) \leq \dots \leq  f(x_1) = f(\gamma_0 s_0).$$
Since we proved that $\gamma_0 \neq 1$ , we have:
$$ f(\gamma_0 s_0) < f(s_0). $$
This implies $f(s_k) < f(s_0)$, that is:
\begin{align*}
   \norme{ \signal - \Phi s_k }^2_2 &< \norme{ \signal - \Phi s_0 }^2_2\\
   \ps{  \Phi s_0, \signal}  &< \ps{ \Phi s_k, \signal}.
\end{align*}
Since $s_0 = \sign(\ps{  \varphi_{i_0}, \signal })\beta e_{i_0}$, both sides of the previous equation are positive:
$$ 0< \ps{ \Phi s_0, \signal}  < \ps{ \Phi s_k, \signal}  $$
This is clearly a contradiction because $s_0 = \argmax_{s \in \Boule_1(\beta)} \ps{  \Phi s, \signal}$. We conclude that $0<\gamma_k < 1$ so that we are again in case (ii) where $\gamma_k=\gamma_k^*$.
\end{itemize}
We conclude that if $\gamma_k>0$ and $\norme{ y }_2 < \beta$ then $\gamma_k = \frac{\ps{ \Phi(s_k - x_k) , r_k }}{\norme{\Phi(s_k - x_k) }^2_2}$.
\end{proof}
%

\begin{proof}[Proof of Theorem~\ref{theorem: exponential convergence FW}]
Note that we are in the case where both Theorem~\ref{theorem: recovery for FW} and Corollary~\ref{cor:convergence} hold.
Let $k$ be an iteration of Algorithm~\ref{algorithm: Frank-Wolfe for ExactSparse}. There are two possibles values for $\gamma_k$:

\paragraph{$\gamma_k = 0$} then $x_k = x_{k+1}$ and subsequently for all $l \geq k$, $x_{k} = x_l$ and $f(x_{k}) = f(x_l)$. 
The convergence of the objective values yields: $f(x_{l}) = f(x^*)=0$ for $l \geq k$
and in particular $\norme{ r_{k+1} }^2_2 = \norme{ r_{k} }^2_2 = 0$. Thus Theorem~\ref{theorem: exponential convergence FW} holds.

\paragraph{ $0<\gamma_k \leq 1$} by definition of the residual, we have:
\begin{align*} 
     \norme{ r_{k+1} }^2_2 & =  \norme{ \signal - \Phi x_{k+1} }^2_2 \\ 
     & =   \norme{ r_k - \Phi\gamma_k(s_k - x_k) }_2^2  \label{rk+1} \\
     & =   \norme{ r_k}^2_2 -2 \gamma_k \ps{ \Phi(s_k - x_k)), r_k } + \gamma_k^2 \norme{\Phi(s_k - x_k)) }_2^2. 
\end{align*}
$\gamma_k$ is the solution of the following optimization problem:
$$ \gamma_k = \argmin_{\gamma \in [0,1]} \norme{ \signal - \Phi (x_k + \gamma(s_k - x_k)) }_2^2. $$
Notice that $\norme{\Phi v}_2\leq \norme{v}_1$ for all vectors $v$ in $\mathbb{R}^n$ since the $\varphi_i$ are of unit norm 
($\norme{\Phi v}_2^2 = \sum_{i,j=1}^n v[i]v[j]\ps{ \varphi_i,\varphi_j}
\leq \sum_{i,j=1}^n |v[i]| |v[j]| =\norme{v}_1^2
$).
Hence $$\norme{y}_2 = \norme{\Phi x^*}_2\leq\norme{x^*}_1 <\beta.$$ 
As showed in Lemma~\ref{lemma: expression of gamma}, the solution of the previous problem is then:
\begin{align*}
  \gamma_k  = & \frac{\ps{ \Phi(s_k - x_k), r_k }}{\norme{\Phi(s_k - x_k) }^2_2}. 
\end{align*}
Replacing the value of $\gamma_k$ in the previous equation, we obtain:
\begin{eqnarray}
\norme{ r_{k+1} }^2_2 = \norme{ r_{k} }^2_2 - \frac{ \ps{  \Phi(s_k - x_k),r_k }^2}{\norme{\Phi(s_k - x_k) }^2_2}. \label{rk+1 apres avoir remplacer gama}
\end{eqnarray}
We shall now lower-bound $\ps{  \Phi(s_k - x_k),r_k }^2$ and upper-bound
 $\norme{\Phi(s_k - x_k) }^2_2$.

To bound $\norme{\Phi (s_k - x_k)}_2^2$, we use $\norme{\Phi v}_2\leq \norme{v}_1$ and $s_k$ and $x_k$ are in $\Boule_1(\beta)$:
\begin{equation}
  \label{borne norme}
  \norme{\Phi (s_k - x_k)}_2^2 \leq\norme{s_k - x_k}_1^2 \leq   4 \beta^2.
\end{equation} 
To bound $\ps{  \Phi(s_k - x_k),r_k }^2$, fix $\epsilon = \tfrac{\beta - \norme{x^*}_1}{2} >0$. 
As noted in Corollary~\ref{cor:convergence}, the iterates $x_k$ converge to $x^*$, there exists an iteration $K$ such that for every $k \geq K$: $ \norme{x_k - x^*}_1 \leq \epsilon $. Fix $k\geq K$.
Let us define $u \in \mathbb{R}^n$, such that:
$$
u[i] = \left\{
    \begin{array}{ll}
        \frac{\epsilon}{\sqrt{m}\norme{\Phi_{\Lambda_{opt}}^tr_k}_2}(\Phi^t r_k)[i] & \mbox{if } i \in \Lambda_{opt} \\
        0 & \mbox{otherwise.}
    \end{array}
\right.
$$
One can show that $x_k + u$ is in $\Boule_1(\beta)$, indeed
\begin{eqnarray*}
 \norme{x_k + u}_1 &\leq& \norme{x_k-x^*}_1 + \norme{x^*}_1 + \norme{u}_1 \\                   
&\leq& \epsilon + \norme{x^*}_1 + \frac{\epsilon}{\sqrt{m}\norme{\Phi_{\Lambda_{opt}}^tr_k}_2}\norme{\Phi_{\Lambda_{opt}}^tr_k}_1 .\\
\end{eqnarray*}
Noting that  $\norme{\Phi_{\Lambda_{opt}}^t r_k}_1 \leq \sqrt{m} \norme{\Phi_{\Lambda_{opt}}^t r_k}_2$ (because $\Phi_{\Lambda_{opt}}^t r_k\in\mathbb{R}^{|\Lambda_{opt}|}$ and $|\Lambda_{opt}|\leq m$) leads to:
$$ \norme{x_k + u}_1\leq 2\epsilon + \norme{x^*}_1=\beta.$$
We conclude that $x_k + u$ is in $\Boule_1(\beta)$. 
Since $s_k = \argmin_{s \in \Boule_1(\beta)}\ps{  s , \nabla f(x_k)} $ then:
\begin{eqnarray*}
 \ps{  s_k , \nabla f(x_k)} &\leq& \ps{   x_k + u , \nabla f(x_k)} \\
\ps{  s_k -x_k , \nabla f(x_k)} &\leq& \ps{   u , \nabla f(x_k)} .
\end{eqnarray*}
By definition of $f$: $\nabla f(x_k) = -\Phi^t r_k $, thus:
\begin{eqnarray*}
\ps{  s_k -x_k , \Phi^t r_k} &\geq& \ps{   u , \Phi^t r_k} \\
\ps{  \Phi (s_k -x_k) ,  r_k} &\geq& \ps{   u , \Phi^t r_k} .
\end{eqnarray*}
Noting that 
$$ \ps{   u , \Phi^t r_k } = \sum_{i\in\Lambda_{opt}}\frac{\epsilon}{\sqrt{m}\norme{\Phi_{\Lambda_{opt}}^tr_k}_2}(\Phi^tr_k)^2[i] = \frac{\epsilon\norme{\Phi_{\Lambda_{opt}}^tr_k}_2}{\sqrt{m}},$$
we conclude:
$$\ps{  \Phi(s_k - x_k), r_k } \geq \frac{\epsilon}{\sqrt{m}} \norme{\Phi_{\Lambda_{opt}}^tr_k}_2.$$
By Theorem~\ref{theorem: recovery for FW}, $r_k$ is in $\Span(\Phi_{\Lambda_{opt}})$ and since the atoms indexed by $\Lambda_{opt}$ are linearly independent (thus $\lambda_{min}^{*} > 0$), we obtain:
$$ \norme{\Phi_{\Lambda_{opt}}^t r_k}_2 \geq \lambda_{min}^{*} \norme{r_k}_2  $$
and
$$ \ps{  \Phi(s_k - x_k), r_k } \geq \frac{\epsilon}{\sqrt{m}} \lambda_{min}^{*} \norme{r_k}_2. $$
By Lemma $2.3$ of~\cite{T04}, $(\lambda_{min}^{*})^2 \geq 1 - \mu_1(m - 1)$.
Since $ m < \tfrac{1}{2}(\mu^{-1} + 1) $ implies $\mu_1(m) + \mu_1(m-1) < 1$~\cite{T04}, we have $0<1 - \mu_1(m-1)<1$ and deduce that:
\begin{eqnarray}
  \ps{ \Phi(s_k - x_k), r_k } \geq \epsilon \sqrt{\frac{1 - \mu_1(m - 1)} {m} } \norme{r_k}_2.  
\label{borne produit scalaire}
\end{eqnarray} 
Plugging in the bounds of Eq.~\eqref{borne norme} and~\eqref{borne produit scalaire} in Eq.~\eqref{rk+1 apres avoir remplacer gama}, we obtain:
\begin{align*}
\norme{ r_{k+1} }^2  = & \norme{ r_{k} }^2 - \frac{\ps{ \Phi(s_k - x_k) , r_k }^2}{\norme{\Phi(s_k - x_k) }^2}  \\
\norme{ r_{k+1} }^2  \leq & \norme{ r_{k} }^2 \left( 1 - \frac{\epsilon^2(1 - \mu_1(m-1))}{4\beta^2m} \right) \\
\norme{ r_{k+1} }^2  \leq & \norme{ r_{k} }^2 \left( 1 - \tfrac{1}{16}\left(\tfrac{1 - \mu_1(m-1)}{m}\right)\left(1-\tfrac{\norme{x^*}_1}{\beta}\right)^2 \right). 
\end{align*}
which finishes the proof.
\end{proof}

\section{Proof of Theorem~\ref{theorem: minimal value for beta}}
\label{sec:proofTh3}
To prove Theorem~\ref{theorem: minimal value for beta}, the key is to bound uniformly the $l_1$ norm of the iterates $x_k$. This is the purpose of the following lemma.
\begin{Lemma} \label{lemma: $x_k$ bounded by $x^*$}
Let $\Phi $ be a dictionary, $\mu$ its coherence, and $\signal = \Phi x^* $ an $m$-sparse signal. If $ m < \tfrac{1}{2}(\mu^{-1} + 1) $, then for each iteration $k$ of Algorithm~\ref{algorithm: Frank-Wolfe for ExactSparse}
$$ \norme{ x_{k} }_1 \leq 2\norme{\signal}_2\sqrt{\tfrac{m}{1 - \mu_1(m-1)}}.$$
\end{Lemma}

\begin{proof}

Indeed, we have on the one hand:
$  \norme{x_k}_1 \leq\sqrt{m}\norme{x_k}_2$  because $x_k$ has non-zero coefficients only in $\lambda_{opt}$ (proved in Theorem~\ref{theorem: recovery for FW}). On the other hand:
\begin{equation*}
\norme{\Phi x_k}_2 \geq \lambda^*_{min}\norme{x_k}_2 \geq \sqrt{1 - \mu_1(m-1)}\norme{x_k}_2.
\end{equation*}
We conclude
\begin{equation*}
\norme{x_k}_1 \leq\sqrt{\tfrac{m}{1 - \mu_1(m-1)}}\norme{\Phi x_k}_2.
\end{equation*}
Moreover
\begin{align*}
\norme{\Phi x_k}_2 &\leq \norme{\Phi x_k-y}_2 + \norme{y}_2\\
&\leq \sqrt{2f(x_k)} + \norme{y}_2\\
&\leq \sqrt{2f(x_0)} + \norme{y}_2\\
&\leq \norme{y}_2 + \norme{y}_2\\
&\leq 2\norme{y}_2,
\end{align*}
where the third line holds because by construction of the Frank-Wolfe algorithm, the sequence $\{f(x_k)\}_k$ is non increasing.
So we conclude that $\norme{x_k}_1 \leq 2\norme{y}_2\sqrt{\tfrac{m}{1 - \mu_1(m-1)}}$ for all $k$.
\end{proof}

\begin{proof}[Proof of Theorem~\ref{theorem: minimal value for beta}]
Since $\beta > 2\norme{\signal}_2\sqrt{\tfrac{m}{1 - \mu_1(m-1)}} \geq\norme{x^*}_1$ (see Remark~\ref{rk:hyp}), we can re-use the arguments of the proof of Theorem~\ref{theorem: exponential convergence FW}. We will do so up to Eq.~\eqref{borne norme} and only modify the lower bound on
$\ps{ \Phi(s_k - x_k) , r_k }^2 $.

By definition of $s_k$,
$$ \ps{ \Phi s_k , r_k } = \beta\max_{i} \va{\ps{  \varphi_i, r_k }} $$
and
\begin{eqnarray*}
\ps{  \Phi x_k ,r_k} & = & \ps{ x_k,  \Phi^t r_k } = \sum_i(x_k)[i] (\Phi^t r_k)[i] \\  
                              & \leq & \norme{x_k}_1 \max_i \va{(\Phi^t r_k)[i]}\\
                              & = & \norme{x_k}_1 \max_i \va{ \ps{  \varphi_i,r_k } }.
\end{eqnarray*}
We conclude that:
\begin{eqnarray*}
\label{eq norme rk et phi sk et xk}
 \ps{\Phi(s_k - x_k),  r_k} & = & \ps{ \Phi s_k , r_k } - \ps{  \Phi x_k, r_k } \\
 & \geq & \beta \max_i \va{ \ps{  \varphi_i, r_k} } \left( 1 - \frac{\norme{x_k}_1}{\beta} \right).
\end{eqnarray*}
By Lemma 2 of~\cite{GV06}:
$$ \max_i \va{ \ps{ \varphi_i, r_k } } \geq \norme{r_k}_2\sqrt{\tfrac{1 - \mu_1(m-1)}{m}}. $$
Hence:
\begin{equation}
\label{eq norme rk et phi sk et xk 2}
 \ps{ \Phi(s_k - x_k), r_k } \geq \beta \norme{r_k}_2\sqrt{\tfrac{1 - \mu_1(m-1)}{m}} \left( 1 - \frac{\norme{x_k}_1}{\beta} \right).
\end{equation}

What is left to prove is that $1 - \frac{\norme{x_k}_1}{\beta}$ is uniformly bounded away from zero.
Using Lemma~\ref{lemma: $x_k$ bounded by $x^*$} we obtain: 
\begin{equation}
\label{eq borne norme 1 xk}
 1- \frac{\norme{x_k}_1}{\beta}  \geq  1- \frac{ 2\norme{\signal}_2}{\beta}\sqrt{\tfrac{m}{1 - \mu_1(m-1)}}.
\end{equation}
By assumption  $1- \frac{ 2\norme{\signal}_2}{\beta}\sqrt{\tfrac{m}{1 - \mu_1(m-1)}}>0$, we deduce that for all $k$:
\begin{equation}
 \ps{ \Phi(s_k - x_k), r_k } \geq \beta \norme{r_k}_2\sqrt{\tfrac{1 - \mu_1(m-1)}{m}} \left( 1-\tau\right)>0,
\end{equation}
with
\begin{equation}
\tau =  \tfrac{ 2\norme{\signal}_2}{\beta}\sqrt{\tfrac{m}{1 - \mu_1(m-1)}}.
\end{equation}
Using this and Eq.~\eqref{rk+1 apres avoir remplacer gama}, we obtain:
\begin{eqnarray*}
\norme{r_{k+1} }^2_2 & = & \norme{ r_{k} }_2^2 - \frac{ \ps{ \Phi(s_k - x_k), r_k }^2}{\|\Phi(s_k - x_k) \|^2}\\
                   & \leq & \norme{ r_{k} }_2^2 - \norme{r_k}_2^2\frac{1 - \mu_1(m-1)}{4m} \left( 1 - \tau\right)^2.
\end{eqnarray*}
We conclude that for all $k$:
$$\norme{r_{k+1} }^2_2 \leq  \norme{ r_{k} }^2_2 \left( 1 - \frac{(1-\mu_1(m-1))}{4m} \left( 1 - \tau\right)^2 \right),$$
with $0<\left( 1 - \frac{(1-\mu_1(m-1))}{4m}\left( 1 - \tau\right)^2 \right)<1$ which proves the exponential convergence from the first iteration. 
\end{proof}

\section*{Acknowledgment}

The authors would like to thank the Agence Nationale de la Recherche under grant JCJC MAD (ANR-14-CE27-0002) which supported this work.

\ifCLASSOPTIONcaptionsoff
  \newpage
\fi



%

\bibliographystyle{plain}
\bibliography{biblio}




%




\begin{IEEEbiographynophoto}{Farah Cherfaoui}
received a B.Sc. degree in computer science in 2016 and a M.Sc. degree in computer science in 2018, both from Aix-Marseille university, Marseille, France. She is now a Ph.D. student at the same university. Her research focuses on active learning and she is also interested in some signal processing problems.
\end{IEEEbiographynophoto}
\begin{IEEEbiographynophoto}{Valentin Emiya} 
graduated from Telecom Bretagne, Brest, France, in 2003 and received the M.Sc. degree in Acoustics, Signal Processing and Computer Science Applied to Music (ATIAM) at Ircam, France, in 2004. He received his Ph.D. degree in Signal Processing in 2008 at Telecom ParisTech, Paris, France. From 2008 to 2011, he was a post-doctoral researcher with the METISS group at INRIA, Centre Inria Rennes - Bretagne Atlantique, Rennes, France. He is now assistant professor (Ma\^itre de Conf\'erences) in computer science at Aix-Marseille University, Marseille, France, within the QARMA team at Laboratoire d'Informatique et Syst\`emes.
His research interests focus on machine learning and audio signal processing, including sparse representations, sound modeling, inpainting, source separation, quality assessment and applications to music and speech.
\end{IEEEbiographynophoto}
\begin{IEEEbiographynophoto}{Liva Ralaivola}
has been Full Professor in Computer Science at Aix-Marseille University (AMU) since 2011 — LIS Research Unit —, member of Institut Universitaire de France since 2016 and part-time researcher at Criteo since 2018. He received the Ph.D. degree in Computer Science from Université Paris 6 in 2003, and his Habilitation à Diriger des Recherches in Computer Science from AMU in 2010. His research focuses on statistical and algorithmic aspects of machine learning with a focus on theoretical issues (risk of predictors, concentration inequalities in non-IID settings, bandit problems, greedy methods). 
\end{IEEEbiographynophoto}
\begin{IEEEbiographynophoto}{Sandrine Anthoine}
received a B.Sc. and M.Sc. in mathematics from ENS Cachan, France, in 1999 and 2001 respectively. She obtained a Ph.D in applied mathematics at Princeton University in 2005. She then worked as a postdoctoral researcher at Aix-Marseille University for a year. Since 2006, she is working as a research fellow for the \emph{Centre National de Recherche Scientifique (CNRS)}. She started at the computer science and signal processing lab. (I3S) of Universit\'e de Nice C\^ote d'Azur and is with the \emph{Institut de Math\'ematiques de Marseille} since 2010.
\end{IEEEbiographynophoto}


\end{document}